\documentclass[11pt]{article}
\usepackage{amsmath,amsthm,amsfonts,amssymb,mathrsfs,bm}
\usepackage{amsmath,amsthm,amsfonts,amssymb,bm,wasysym}
\usepackage{epsfig}
\usepackage[usenames]{color}
\usepackage{verbatim}
\usepackage{hyperref}
\usepackage{multicol}
\usepackage{comment}
\usepackage{float}
\usepackage{color}
\usepackage{enumerate}
\usepackage[normalem]{ulem}

\usepackage[color=green, textsize=tiny]{todonotes}

\parskip \medskipamount
\newtheorem{theorem}{Theorem}[section]

\newtheorem{lemma}[theorem]{Lemma}
\newtheorem{proposition}[theorem]{Proposition}

\newtheorem{remark}[theorem]{Remark}

\newtheorem{claim}[theorem]{Claim}
\newtheorem{conjecture}[theorem]{Conjecture}

\numberwithin{equation}{section}

\def\N{\mathbb{N}}
\def\Z{\mathbb{Z}}

\def\R{\mathbb{R}}

\def\bP{\mathbb{P}}
\def\F{\mathcal{F}}

\newcommand{\cpc}[2]{\mathrm{BCap}_{#1}(#2)}
\newcommand{\cp}[2]{\mathrm{Cap}_{#1}(#2)}

\def\cU{\mathcal{U}}
\def\cT{\mathcal{T}}

\def\cR{\mathcal{R}}

\def\cF{\mathcal{F}}
\def\cE{\mathcal{E}}

\def\cA{\mathcal{A}}

\def\reff#1{(\ref{#1})}

\renewcommand{\phi}{\varphi}
\renewcommand{\epsilon}{\varepsilon}

\allowdisplaybreaks

\newcommand{\1}{{\text{\Large $\mathfrak 1$}}}

\newcommand{\til}{\widetilde}

\newcommand{\pr}[1]{\mathbb{P}\!\left(#1\right)}
\newcommand{\E}[1]{\mathbb{E}\!\left[#1\right]}

\newcommand{\prcond}[3]{\mathbb{P}_{#3}\!\left(#1\;\middle\vert\;#2\right)}
\newcommand{\econd}[2]{\mathbb{E}\!\left[#1\;\middle\vert\;#2\right]}

\begin{document}

\title{\bf Branching random walks and Minkowski sum of random walks}

\author{Amine Asselah \thanks{
Universit\'e Paris-Est, LAMA, UMR 8050, UPEC, UPEMLV, CNRS, F-94010 Cr\'eteil; amine.asselah@u-pec.fr} \and
Izumi Okada\thanks{Department of Mathematics and Informatics, Faculty of Science, Chiba University, Chiba 263-8522, Japan; 
iokada@math.s.chiba-u.ac.jp} \and 
Bruno Schapira\thanks{Aix-Marseille Universit\'e, CNRS, I2M, UMR 7373, 13453 Marseille, France;  bruno.schapira@univ-amu.fr} \and Perla Sousi\thanks{University of Cambridge, Cambridge, UK;   p.sousi@statslab.cam.ac.uk} 
}

\date{}
\maketitle

\begin{abstract} We show that the range of a critical branching random walk conditioned to survive forever and the Minkowski sum of two independent simple random walk ranges are {\it intersection-equivalent} in any dimension $d\ge 5$, in the sense that they hit any finite set with comparable probability, as their common starting point is sufficiently far away from the set to be hit.
Furthermore, we extend a discrete version of Kesten, Spitzer and Whitman's result on the law of large numbers for the volume of a {\it Wiener sausage}. Here, the sausage is made of the Minkowski sum of $N$  independent simple
random walk ranges in $\mathbb Z^d$, with $d>2N$, and of a finite set $A\subset \Z^d$. When properly normalised 
the volume of the sausage converges to a quantity equivalent to the capacity of $A$ with respect to the kernel $K(x,y)=(1+\|x-y\|)^{2N-d}$. As a consequence, we establish a new relation between capacity and {\it branching capacity}. 
\newline
\newline
\emph{Keywords and phrases.} Capacity, range of random walk, branching random walk, branching capacity, intersection probability.
\newline
MSC 2020 \emph{subject classifications.} 60F15; 60G50; 60J45.
\end{abstract}

\section{Introduction}
In this paper we establish similarities between the typical behaviour of two multi-parameter
processes whose Green's functions are comparable: the
Minkowski sum of two independent random walks and
the infinite invariant critical branching random walk. Both processes
are considered in the transient regime on $\mathbb Z^d$, that is when $d>4$.
The analogy holds first at the level of the volume of their {\it Wiener sausages} associated with
any set $A\subset \Z^d$. More precisely, the
Wiener sausage (of a trajectory of the process) is obtained as we {\it roll a finite set} $A\subset \mathbb Z^d$
over the trajectory. 
Secondly, the analogy holds for hitting times {\it from infinity}, showing some form of intersection-equivalence,
a notion first discussed by notion first discussed by Benjamini, Pemantle and Peres~\cite{BPP95}.
We then consider the Minkowski sum of $N$ independent random walks,
both in terms of their Wiener sausages, and then in terms of their hitting times.
Finally, in the critical dimension $d=4$, we provide a law of large
numbers result for the capacity of a discrete Wiener sausage.

\paragraph{The models.}
We first consider a critical branching random walk, and its 
{\it infinite} version. In this process, time is indexed
by a random tree obtained as follows:
let $\mu$ be an offspring distribution with mean one and positive finite 
variance $\sigma^2$, and denote by $\cT_c$ the corresponding critical Bienaym\'e-Galton-Watson (BGW) tree with root $\emptyset$. Next, attach independent increments to the edges of the tree, 
whose common law is taken for simplicity to be the uniform distribution on the neighbours of the origin. 
For $x\in \Z^d$, a branching random walk starting from $x$ 
is obtained by assigning to each vertex $u\in \cT_c$ the sum of the 
increments along the edges on the shortest path joining $u$ to~$\emptyset$, and translating the
random cloud by $x$. We denote by $\cT_c^x$ the set of vertices of $\Z^d$ visited by the nodes of $\cT_c$ when the root starts from $x$. Since the tree is critical, its Green's function, say $g$, is equal to the 
Green's function of a simple random walk on $\Z^d$. 
However, as we condition on the genealogy to be infinite and reroot appropriately, we obtain an infinite tree $\cT$ which is an example of 
Aldous' invariant sin-trees~\cite{Ald91}:
it is made of a {\it spine}, that is a semi-infinite line of nodes $(\emptyset,u_1,\dots)$,
and to each of its nodes we attach a critical tree built as follows. Independently, each node $u_i$ for $i>0$
draws a random number of children $Z_i$ with size-biased 
distribution $\mu_{\textrm{sb}}$ (defined by $\mu_{\textrm{sb}}(i ) =i\mu (i)$), identifying one uniformly at
random with $u_{i+1}$, and thus partitioning 
the $Z_i-1$ other children as left  and right on each side of the spine,
and letting each of them in turn produce a critical BGW tree with reproduction law $\mu$. We attach $Z_0$ children to the root $\emptyset$ distributed according to 
$\widetilde \mu(i)=\mu(i-1)$ (for $i\ge 1$) instead of $\mu_{\textrm{sb}}$. Its first child is identified with $u_1$, and the remaining 
$Z_0-1$ produce in turn and independently $\mu$-critical trees.
The root is assigned label $0$ and, using a clockwise depth-first search algorithm from the root, 
we label vertices on the right of the spine with positive labels.  
Using a counter-clockwise depth-first search, we label vertices on the left of the spine with negative labels. 
The set of vertices with positive labels is called the future of $\mathcal T$, and is denoted by $\mathcal T_+$, 
while the set of vertices with negative labels is called the past of $\mathcal T$, and is denoted by $\mathcal T_-$. Again we assign independent simple random walk increments to the edges of the tree. When the root starts from $x$, we write $\cT^x$ for the range of the tree $\cT$ and $\cT_-^x$ (resp.\ $\cT_+^x$) for the range of the past (resp.\ future).
The Green's function of the corresponding infinite tree-indexed walk is of the order of $g*g$, i.e.\ 
the convolution of $g$ with itself, which is finite exactly when $d\ge 5$. 

The second model is the {\it Minkowski sum} of two random walk ranges.
We recall that the Minkowski sum of two subsets 
$A,B\subset \mathbb Z^d$, is defined to be $A+B = \{a+b: a\in A, b\in B\}$.
Let $(X_n)_{n\ge 0}$ and $(\til{X}_n)_{n\geq 0}$ be two independent simple random walks on $\Z^d$. We denote by $\cR_\infty = \{X_n : n\ge 0\}$ and $\til{\cR}_\infty= \{\til{X}_n : n\ge 0\}$ their respective ranges. In this paper we study their Minkowski sum which is simply $\cR_\infty+\til{\cR}_\infty$. Thus, intuitively speaking one rolls on the support of one walk the 
support of another independent walk, obtaining a {\it sausage}. 

\paragraph{Wiener sausage.}
A celebrated result of Kesten, Spitzer and Whitman~\cite[p.\ 252]{IMcK74} (see also \cite{Sp64}) concerns
the volume of a Wiener sausage obtained as we roll a compact set, say $A\in \R^d$ with $d>2$, over
the trajectory of a transient Brownian trajectory. As we run the sausage over a time period of length~$t$, and
divide the volume of the sausage by $t$, the ratio converges to the electrostatic capacity of the set $A$.
 In our discrete setting their result reads as follows. 
Given a simple random walk $(X_n)_{n\ge 0}$, 
define its range in the time window $[a,b]$, with $0\le a\le b\le \infty$, 
by $\mathcal R[a,b]= \{X_a,\dots,X_b\}$, with the short-hand notation $\mathcal{R}_n = \mathcal R[0,n]$. 
Then almost surely, for any finite set $A\subset \mathbb Z^d$, with $d\ge 3$, 
\begin{equation}\label{Whit.Sp}
\lim_{n\to \infty} \frac{|\mathcal R_n+A|}{n} =\cp{}{A}.
\end{equation}
The limiting functional $\cp{}{A}$ turns out to be the discrete capacity of $A$. It is linked with
the Green's function through an {\it energy}.  
Indeed, for a kernel $K:\mathbb Z^d\times \mathbb Z^d\to\mathbb R^+$, 
and a probability measure $\nu$ on $A$, the $K$-energy is defined to be 
\[
\cE_{K}(\nu):=\sum_{x\in A}\sum_{y\in A}  K(y-x) \nu(x) \nu(y).
\] 
The capacity of the set $A$ is defined as
\begin{equation}\label{def-capacity}
\cp{}{A} :=\Big(\inf \Big\{\cE_{g}(\nu) :\ \nu \textrm{ probability measure on }A\Big\}\Big)^{-1}.
\end{equation}
As Spitzer observed later~\cite{Sp73},~\eqref{Whit.Sp} follows directly from Kingman's subadditive ergodic theorem. Now, consider the range $\cT^0_n$ of the first $n$ sites of
the walk indexed by $\mathcal T$ starting from the origin. 
Zhu showed in~\cite{Zhu16} that, when $d>4$, the hitting probability of a finite set, say
$A\subset \mathbb Z^d$, by the past
tree $\cT_-$ appropriately normalised has a limit that he called {\it the branching capacity} of $A$. More precisely,
\begin{equation}\label{hitting-zhu}
\textrm{BCap}(A):=\lim_{\|x\|\to\infty} \frac{2/\sigma^2}{ g*g(x)} \mathbb P\big(\mathcal T_-^x \cap A\neq \emptyset \big). 
\end{equation}
Furthermore, it was shown in \cite{ASS23} that
the branching capacity is comparable to a capacity corresponding to the kernel $g*g$ in the following sense: there exists a positive constant $C$ depending on the variance of the offspring distribution $\mu$, so that for all finite sets $A\subset \Z^d$, we have 
\begin{equation}\label{abp-var}
\frac{1}{C} \cdot \textrm{BCap}(A)^{-1}\leq \inf\Big\{ \cE_{g*g}(\nu):\ \nu \textrm{ probability measure on }A\Big\} \leq C\cdot 
\textrm{BCap}(A)^{-1}.
\end{equation}

Let $\mathcal{A}$ be an arbitrary set. For any two functions $f,h: \mathcal{A}\to \R_+$ we write $f\asymp h$ if the ratio $f(a)/h(a)$ is bounded both from above and below by positive constants uniformly over all $a\in \mathcal{A}$.

Our first result extends the result by Kesten, Spitzer and Whitman to tree-indexed random walks and additive random walks, thus revealing similarities between these two processes. Fix an offspring distribution $\mu$ with mean one and finite variance and consider the associated infinite tree $\mathcal T$. 
%
%
%

\begin{theorem}\label{theo-2saus}
Fix $d\geq 5$. Let $X$ and $\til{X}$ be two independent simple random walks on $\Z^d$ and let~$\cT_n^0$ be the first $n$ sites of a walk in $\Z^d$ indexed by the tree $\cT$ with the root starting from $0$. Let $A$ be a finite subset of $\Z^d$. Then
the following limits hold almost surely,
\begin{equation}\label{2saus-3}
 \lim_{n\to \infty} \frac{|\mathcal T^0_n+A|}{n}
\asymp 
\lim_{n\to \infty} \frac{|\cR_n+\widetilde \cR_n+A|}{n^2},
\end{equation}
with the implied constants only depending on the variance of $\mu$. 
Moreover, 
\begin{equation}\label{2saus-1}
 \lim_{n\to \infty} \frac{|\mathcal T^0_n+A|}{n} =
{\rm{BCap}}(A),
\end{equation}
and 
\begin{equation}\label{2saus-2}
\lim_{n\to \infty} \frac{|\cR_n+\widetilde \cR_n+A|}{n^2}=
\lim_{n\to \infty} \frac{\cp{}{\cR_n+A}}{n}.
\end{equation}  
\end{theorem}
Theorem~\ref{theo-2saus} states that two limits in \reff{2saus-3} exist, and
that they are comparable. 
To the best of our knowledge, the observation that the limits are comparable is new, as well
as the characterisation of the limits in terms of the branching capacity. 
Note that \reff{2saus-2} generalises a law of large numbers result for $\cp{}{\cR_n}$ proved by Jain and Orey
\cite{JO69} with the limit being positive if and only if~$d>4$. 
In Lemma~\ref{lem-dual} we give a representation of the limit in terms of escape probabilities in analogy with the classical formula for capacity. We finally note that~\reff{2saus-1} generalises a recent result of Le Gall and Lin~\cite{LGL16} where they treat
the case $A=\{0\}$. More precisely, Le Gall and Lin study a critical BRW conditioned on having exactly $n$ nodes and take the limit as $n\to\infty$. In order to establish limit laws for the volume of its range, they had to introduce the infinite invariant tree we described above as its invariance under the shift of labels makes the corresponding law of large numbers a simple application of Kingman's subadditive ergodic theorem. Here, our starting point
is directly the infinite invariant tree, and the interested reader can check~\cite{LGL16} for linking the
conditioned process and the infinite invariant one. Note that in the case $A=\{0\}$, Le Gall and Lin
provide a dual formula for ${\rm{Bcap}}(\{0\})$ in Theorem 4 of~\cite{LGL16}. This formula was later generalised by  Zhu in~\cite{Zhu16}, where he defined the branching capacity of any finite set $A\subset \Z^d$, as in~\eqref{hitting-zhu}.

\paragraph{Hitting probabilities.}
 Our second result offers another characterisation of hitting probabilities for infinite branching random walks.
In the terminology of~\cite{P96}, we now show that
$\cT_-^x$ and $x+\mathcal R_\infty+\widetilde{\mathcal R}_\infty$ are {\it intersection-equivalent} when $x$ is away from the set to be hit.
\begin{theorem}\label{theo.IE}
Assume $d\ge 5$ and let $\mu$ be an offspring distribution with mean one and finite variance. There exists a positive constant~$C$ so that the following holds. Let $\cR_\infty$ and $\til{\cR}_\infty$ be two independent simple random walk ranges in $\Z^d$. Then for any finite set $A\subset \Z^d$, and any $x$ sufficiently far  from $A$, we have,
\begin{equation}\label{inter-equiv1}
\frac{1}{C}\cdot \mathbb P\big(\mathcal T_-^x \cap A\neq \emptyset \big) \le \mathbb P\Big((x+\mathcal R_\infty+\widetilde{\mathcal R}_\infty) \cap A\neq \emptyset \Big) 
\le C\cdot \mathbb P\big(\mathcal T_-^x \cap A\neq \emptyset \big).
\end{equation}
\end{theorem}
Note that Benjamini, Pemantle and Peres~\cite{BPP95}
focus on the Martin capacity for a general
transient Markov chain on a countable state space. The Martin capacity is associated
with the so-called Martin Kernel $K(x,y)=g(x,y)/g(\rho,y)$, where $\rho$ is the starting site of the chain and $g$ is its Green's function. 
Their main result states that the hitting probability of a set $A$ is within a constant factor of two equal to the Martin capacity of $A$; in particular two Markov chains with comparable Green's functions are intersection equivalent.
The main difficulty in proving Theorem~\ref{theo.IE} is of course the lack of Markov property:
for branching random walk or for additive random walks, 
as Salisbury explains in \cite{S96} {\it "the difficulty is that there may be
no first hitting time at which any kind of Markov property applies"}. Fitzsimmons and Salisbury~\cite{FS89,S96} using
remarkable ideas managed to deal with multivariate processes. 
Here we shall adapt proofs of Khoshnevisan and Shi~\cite{KS99}, inspired by \cite{S96},
which provide estimates on the probability that a
Brownian sheet hits a distant compact set $A\subset \mathbb R^d$, 
in terms of some appropriate (continuous) $\gamma$-capacity of $A$.

\paragraph{Sum of $N$ random walks.} 
One of the achievements of potential theory for Markov processes (in the continuous setting) is to establish 
necessary and sufficient conditions for a transient process to hit a set. Morally, the set should
support a probability measure whose $g$-energy is finite, where~$g$ is the Green's function.
Fitzsimmons and Salisbury in \cite{FS89} managed to build such a measure for
additive Markov processes using random times which are not stopping times.
Their proof requires estimates on the first two moments of the local times and 
their approach is well adapted to tackle a sum of $N$ walks.
The local times process of the sum of~$N$ independent walks $\{X^1,\dots,X^N\}$, is defined as follows: given any $x,z\in \mathbb Z^d$, 
\begin{equation}\label{def-localN}
\ell_{z+X^1+\dots+X^N}(x):=
\sum_{t_1,\dots,t_N} \1\big(z+X^1_{t_1}+\dots+X^N_{t_N}=x\big).
\end{equation}
Let $G_N$ be the $N$-th convolution power of the simple random walk's Green's function $g$
(so that $G_N(x-z)=\E{\ell_{z+X^1+\dots+X^N}(x)}$). Our second moment estimate reads as follows.
\begin{lemma}\label{lem.second.moment}
\emph{
Let $N\ge 1$ and $d>2N$. 
There exists a constant $C=C(d)>0$, such that for any $z,a,b\in \Z^d$, with $\|z\|\ge 2\max(\|a\|,\|b\|)$, we have 
\begin{equation}\label{2moment}
\E{\ell_{z+X^1+\dots+X^N}(a)\ell_{z+X^1+\dots+X^N}(b)}
\leq C\, \Big(G_{N}(z-a) + G_{N}(z-b)\Big)\cdot G_{N}(a-b). 
\end{equation}
}
\end{lemma}
We now define the $\gamma$-capacity, denoted  $\textrm{Cap}_{\gamma}$, for any $\gamma>0$, by replacing $g$ in~\eqref{def-capacity} by the kernel $K_\gamma(x,y)=(1+\|x-y\|)^{-\gamma}$. More precisely, for a finite and nonempty subset $A\subset \mathbb Z^d$, we define
\begin{equation}\label{intro-5}
\frac{1}{\cp{\gamma}{A}}= \inf \Big\{\cE_{K_\gamma}(\nu):\ \nu \textrm{ probability measure on }A\Big\}.
\end{equation}
Our third result is a natural generalisation of~Theorem~\ref{theo-2saus}. 
\begin{theorem}\label{thm.fAN} 
Let $N\ge 1$, and $d>2N$. Let $(\mathcal R_\infty^i)_{i=1,\dots,N}$ be $N$ 
independent simple random walk ranges on $\Z^d$. 
There exists a positive set-function $f_N$, such that for any finite and nonempty $A\subset \mathbb Z^d$, almost surely (with the convention that $f_0(A) = |A|$), 
\begin{equation}\label{intro-6}
f_N(A):=\lim_{n\to\infty}  \frac{\big| \mathcal R_n^1 +\dots+\mathcal R_n^N +A\big|}{ n^N}
 \asymp \cp{d-2N}{A}.
\end{equation}
Furthermore, the following limit exists almost surely and satisfies
\begin{equation}\label{intro-7}
 \lim_{n\to \infty} \frac{\cp{d-2(N-1)}{\mathcal R_n^1+A}}{n}\asymp \cp{d-2N}{A}.
\end{equation}
\end{theorem}

\begin{remark}
	\rm{
	In Section~\ref{sec:multiparameter} we give a dual representation of $f_N(A)$ in terms of escape probabilities for sums of walks analogously to the case of capacity for $N=1$. 
	}
\end{remark}

In our next result we relate  hitting probabilities of a set by the sum of $N$ walks to its $(d-2N)$-capacity.

\begin{theorem}\label{thm.KSd}
Let $N\ge 1$ and $d\ge 1+2N$. Let $(\mathcal R_\infty^i)_{i=1,\dots,N}$, be $N$ independent simple random walk ranges
in $\Z^d$.
There exists a positive constant $C$, such that for any finite set $  A\subset \mathbb Z^d$, containing the origin, and any $z\in \mathbb Z^d$, with $\|z\|\ge 2 \cdot {\rm diam}(A)$,
\begin{equation}\label{hit.Cd-4}
\frac1C\cdot  \frac{\cp{d-2N}{A}}{ \|z\|^{d-2N}} \le  \mathbb P\Big((z + \mathcal R_\infty^1+\dots+\mathcal R_\infty^N) \cap A \neq \emptyset\Big)  \le C\cdot \frac{\cp{d-2N}{A}}{ \|z\|^{d-2N}}.
\end{equation}
\end{theorem}

Note that the case $N=1$ of Theorem~\ref{thm.KSd} is well-known, 
see e.g.~\cite{L91} or~\cite{BPP95} for a more precise result. 
In addition to Theorem~\ref{thm.KSd} we have the following proposition, which makes the link between the probability of hitting a set $A$ for a sum of ranges, and the ergodic limit appearing in \reff{intro-7} of
Theorem~\ref{thm.fAN}.

\begin{proposition}\label{thm.charact}
Let $N\ge 2$ and $d> 2N$. Suppose that $\cR^1,\ldots, \cR^N$, are $N$ independent simple random walk ranges in $\Z^d$. There exist positive constants $c_1$ and $c_2$, 
such that for any finite nonempty set $A\subset \mathbb Z^d$,
\begin{equation}\label{hit.fA}
\liminf_{\|z\|\to \infty} \mathbb \|z\|^{d-2N} \cdot \mathbb P\Big((z + \mathcal R_\infty^1+\dots+\mathcal R_\infty^N) \cap  A \neq \emptyset\Big)\ge c_1\cdot\lim_{n\to \infty} \frac{\cp{d-2(N-1)}{\mathcal R_n^1+A}}{n},
\end{equation}
and
\begin{equation}\label{hit.fA.bis}
 \limsup_{ \| z\| \to \infty} \mathbb \|z\|^{d-2N} \cdot \mathbb P\Big((z + \mathcal R_\infty^1+\dots+\mathcal R_\infty^N) \cap A \neq \emptyset\Big)
   \le c_2\cdot \lim_{n\to \infty} \frac{\cp{d-2(N-1)}{\mathcal R_n^1+A}}{n}.
\end{equation}
\end{proposition}

Observe that~\eqref{intro-7} indeed follows from Theorem~\ref{thm.KSd} and Proposition~\ref{thm.charact}.\paragraph{Critical dimension four.}
We now briefly discuss the case of dimension four which is critical for the capacity of the range. By analogy with the case of the volume in dimension two, first considered in Spitzer's original paper~\cite{Sp64}, and then by Le Gall~\cite{LG90} and Port~\cite{Port65}, one can expect that in the asymptotic development of
$\mathbb E[\cp{}{\mathcal R_n +A}]$, only the second order term should depend on $A$ (and be related to a properly defined notion of branching capacity). 
Here we do not pursue such a precise result,
but notice that indeed the first order term does not depend on $A$. 

\begin{proposition}\label{thm.d4}
Let $\cR$ be a simple random walk range in $\Z^4$. Then for any finite and nonempty set $A\subset \Z^4$, one has 
\begin{equation}\label{d4-capacity}
\lim_{n\to\infty}  \frac{\log n}{n}\E{\cp{}{\mathcal R_n+A}}= \frac{\pi^2}{8}. 
\end{equation}
\end{proposition}

To prove the proposition above we use key ideas from Lawler's book~\cite{L91}: the relationship between capacity and Green's function in Theorem 3.6.1, and the estimates from Section 3.4 in~\cite{L91}. It was proved in~\cite{ASS19} that when $A=\{0\}$, then $(\log n) \cp{}{\cR_n +A}/n$ converges to $\pi^2/8$ almost surely as $n\to\infty$. The same argument as in~\cite{ASS19} can be used to prove almost sure convergence also in the case when $A$ is a general finite set. However, a central limit theorem is missing in the general case. 

\noindent \textbf{Notation.} 
We will use the notation $f\gtrsim g$ if there exists a positive constant $c$, such that $f\ge cg$, and $f\lesssim g$ (or sometimes $f=\mathcal O(g)$) if $g\gtrsim f$. We also use the standard notation $o(1)$ for a quantity which converges to $0$ as the parameter $n$ goes to infinity.
We denote by $\|\cdot \|$ the Euclidean norm, and for $x\in \mathbb Z^d$, and $r\ge 0$ let $B(x,r)=\{y\in \mathbb Z^d:\|y-x\|\le r\}$, the Euclidean ball of radius $r$. We write $\partial \Lambda$ for the inner boundary of a set $\Lambda\subseteq \mathbb Z^d$, i.e.\ the set of points in $\Lambda$ having at least one neighbor in $\Lambda^c$.

The paper is organised as follows. In Section~\ref{sec-subadditive}, we gather known results from ergodic theory that we apply to trees and sums of walks. We then provide an expression for $f_N(A)$ from Theorem~\ref{thm.fAN}, and show that it is positive
when $d>2N$. Finally, we recall why $\gamma$-capacities are sub-additive. In Section~\ref{sec-localtimes} we prove Lemma~\ref{lem.second.moment}, which is an important ingredient for the proof of Theorem~\ref{thm.KSd} given in Section~\ref{sec.KSd}. In Section~\ref{sec.charact} we prove Proposition~\ref{thm.charact} using Theorem~\ref{thm.KSd} and in Section~\ref{sec.d4} we focus on the $4$-dimensional case and give the proof of Proposition~\ref{thm.d4}.
Finally, in Section~\ref{sec-open}, we gather related open problems.

\section{Subadditive functionals \& Ergodic theorems}\label{sec-subadditive}

In this section we show the existence of three ergodic limits, namely the limits in~\eqref{2saus-3}, \eqref{intro-6} and~\eqref{intro-7} holding almost surely. We start in the next section by recalling some results about~$\gamma$-capacities. Then in Section~\ref{sec:multiparameter} we recall a multi-parameter extension of the subadditive ergodic theorem and then deduce that the limit in~\eqref{intro-6} exists almost surely. Then in Section~\ref{sec-tree} we apply it to functionals on trees.

\subsection{$\gamma$-Capacities}\label{sec-cap}
In this section we collect some results about $\gamma$-capacities. 
Note that on the other hand~\eqref{intro-7} directly follows from 
Kingman's subadditive theorem~\cite{K73} and the subadditivity 
of $\gamma$-capacities, which we recall now.

\begin{claim}
    Let $\gamma>0$. Then for any finite sets $A,B \subseteq \Z^d$ we have
        \[
        \cp{\gamma}{A\cup B}\le \cp{\gamma}{A} +\cp{\gamma}{B}.
        \]
\end{claim}

\begin{proof}[\bf Proof]
First notice that $\text{Cap}_\gamma$ is increasing for inclusion, i.e.\ if $A\subseteq B$, then  $\text{Cap}_\gamma(A) \le \text{Cap}_\gamma(B)$, since a probability measure on $A$ is also a probability measure on $B$. It thus suffices to prove sub-additivity for disjoint subsets.
Now consider $A$ and $B$ two disjoint subsets of $\mathbb Z^d$, and let $\nu$ be a probability measure on $A\cup B$. Let $\alpha = \sum_{x\in A} \nu(x)$. Then it is easy to see that
$$\sum_{x,y\in A\cup B} (1+\|x-y\|)^{-\gamma} \nu(x) \nu(y) \ge \frac{\alpha^2}{\text{Cap}_\gamma(A)} + \frac{(1-\alpha)^2}{\text{Cap}_\gamma(B)}.$$ Indeed, this is trivially true if $\alpha\in \{0,1\}$, while otherwise the restriction of $\nu/\alpha$ to $A$ is a probability measure on $A$,
and the restriction of $\tfrac{\nu}{1-\alpha}$ to $B$ is a probability measure on $B$.
Taking the infimum over all $\nu$ on the left hand side yields
$$\frac{1}{\text{Cap}_\gamma(A\cup B)} \ge \inf_{\alpha \in [0,1]} \left\{\frac{\alpha^2}{\text{Cap}_\gamma(A)} + \frac{(1-\alpha)^2}{\text{Cap}_\gamma(B)}\right\}. $$
Now observe that for any $x,y>0$, and any $\alpha\in [0,1]$,
$$\frac{\alpha^2}{x} + \frac{(1-\alpha)^2}{y} \ge \frac 1{x+y},$$
which proves well that $\text{Cap}_\gamma(A\cup B) \le \text{Cap}_\gamma(A) + \text{Cap}_\gamma(B)$ and finishes the proof.
\end{proof}

\begin{lemma}
        Let $d\geq 3$, $\gamma>2$ and $A\subseteq \Z^d$, be a finite set. Let $\cR$ be the range of a simple random walk in $\Z^d$. Then we have almost surely
        \begin{equation*}
\lim_{n\to \infty} \frac{\cp{\gamma}{\mathcal R_n + A}}{n} = \inf_{n\ge 1} \frac{\E{\cp{\gamma}{\mathcal R_n + A}}}{n} \gtrsim \cp{\gamma-2}{A}. 
\end{equation*}
\end{lemma}

\begin{proof}[\bf Proof]

Let $\mu$ be a probability measure on $A$, and consider the probability measure $\nu_n$ on $\mathcal R_n +A$ given by
$$\nu_n(z) = \sum_{x\in \mathcal R_n} \sum_{a\in A} \1(x+a=z)\cdot \frac{\ell_n(x)}{n+1} \cdot \mu(a), \qquad\textrm{for all } z\in \mathcal R_n +A. $$
Note that this is indeed a probability measure on $\mathcal R_n+A$.
Then we have for any $n\geq 1$
\begin{align}\label{eq:ineqforrnplusa}
\begin{split}
        \text{Cap}_\gamma(\mathcal R_n +A) &\geq \frac{1}{\sum_{x,y\in \Z^d} (1+ \|x-y\|)^{-\gamma}\nu_n(x) \nu_n(y)} \\
        &= \frac{(n+1)^2}{\sum_{a,b\in A}\sum_{x,y\in \Z^d} (1+ \|a-b+ x-y\|)^{-\gamma}\mu(a) \mu(b)\ell_n(x)\ell_n(y)}.
        \end{split}
         \end{align}
         By Jensen's inequality we get
         \begin{align*}
                \E{\text{Cap}_\gamma(\mathcal R_n +A)} \geq \frac{(n+1)^2}{\sum_{a,b\in A}\sum_{x,y\in \Z^d} (1+ \|a-b+ x-y\|)^{-\gamma}\mu(a) \mu(b)\E{\ell_n(x)\ell_n(y)}}.
         \end{align*}
Let $g_n(x) = \mathbb E[\ell_n(x)]$, and $g(x) = g_\infty(x)$. Then we bound using the Markov property,
$$\mathbb E[\ell_n(x) \ell_n(y)] \le (g_n(x) + g_n(y))g(x-y).$$
We write $h_\gamma(u) = (1+\|u\|)^{-\gamma}$. Since $g(x-y)\asymp (1+\|x-y\|)^{2-d}$, and $\gamma>2$ by hypothesis, we get for all $a,b\in A$
$$\sum_{x,y \in \mathbb Z^d} h_\gamma(a-b+x-y)\cdot \mathbb E[\ell_n(x)\ell_n(y)] \lesssim (h_\gamma*g(a-b))\cdot\sum_{x\in \mathbb Z^d} g_n(x) = (n+1)\cdot (h_\gamma*g(a-b)).$$
Plugging this into~\eqref{eq:ineqforrnplusa} we get
\begin{align*}
        \text{Cap}_\gamma(\mathcal R_n +A) &\gtrsim \frac {n+1}{\sum_{a,b\in A}  h_\gamma*g(a-b) \mu(a)\mu(b)}
\end{align*}
and it is not difficult to see that as soon as $\gamma>2$, and $d\ge 3$, then $h_\gamma*g \asymp h_{\gamma - 2}$. Therefore taking the infimum above over all probability measures $\mu$ on $A$, we get
\begin{equation}\label{lower.gammacap}
\lim_{n\to \infty} \frac{\text{Cap}_\gamma(\mathcal R_n + A)}{n} = \inf_{n\ge 1} \frac{\mathbb E\big[\text{Cap}_\gamma(\mathcal R_n + A)\big]}{n} \gtrsim \text{Cap}_{\gamma-2}(A)
\end{equation}
and this concludes the proof.
\end{proof}

\subsection{Multiparameter subadditive ergodic theorem}\label{sec:multiparameter}

We start by recalling a multi-parameter extension of the subadditive ergodic theorem due to Akcoglu and Krengel. We denote by $\mathcal U_N$ the set of all $N$-dimensional rectangles of $\mathbb N^N$, i.e.\ sets of the form $\prod_{i=1}^N \{n_i,\dots,m_i\}$, with $0\le n_i\le m_i$ for all~$i\le N$.

\begin{theorem}[Akcoglu--Krengel~\cite{AK81}]\label{thm:akcoglu-krengel}
Let $N\ge 1$, and $(L(U))_{U\in \mathcal U_N}$ be a sequence of real-valued random variables, satisfying the following properties:
\begin{itemize}
\item[$(i)$] (Stationarity) For any $k$, any $U_1,\dots,U_k \in \mathcal U_N$, and any $u\in \mathbb N^N$, the
joint distribution of $(L(u+U_1),\dots,L(u+U_k))$ is the same as that of $(L(U_1),\dots,L(U_k))$.
\item[$(ii)$] (Subadditivity) Given any disjoint rectangles $U_1,\dots,U_k$, such that $\cup_{i=1}^k U_i \in \mathcal U_N$,
one has $L(\cup_{i\le k} U_i) \le \sum_{i \le k} L(U_i)$.
\item[$(iii)$] (Integrability) The random variables $L(U)$ are integrable for all $U\in \mathcal U_N$.
\item[$(iv)$] (Boundedness in mean) One has $\sup_{n\ge 0} \mathbb E\big[|L(\{0,\dots,n\}^N)|\big]/n^N<\infty$.
\end{itemize}
Then there exists $\gamma\in \mathbb R$, such that almost surely
$$\lim_{n\to \infty} \frac{L(\{0,\dots,n\}^N)}{n^N}  = \gamma,$$
and furthermore,
$$\gamma = \inf_{n_1,\dots,n_N\ge 1} \frac{\mathbb E\big[L\big(\prod_{i=1}^N \{0,\dots,n_i\}\big)\big]}{n_1\dots n_N} = \lim_{n_1,\dots,n_N\to \infty} \frac{\mathbb E\big[L\big(\prod_{i=1}^N \{0,\dots,n_i\}\big)\big]}{n_1\dots n_N}. $$
\end{theorem}

Using this we now explain the existence of the limit in~\eqref{intro-6}.

First, note that the volume of a sausage is a Minkowski sum $|\cR+A|$, 
where $\cR$ is the (random) support of our process, and $A$
is a finite subset of $\mathbb Z^d$. The elementary exclusion-inclusion formula gives that
\[
|\cR+(A\cup B)|=|(\cR+A)\cup (\cR+B)|= |\cR+A|+|\cR+B|-|(\cR+A)\cap(\cR+B)|.
\]
Since, $\cR+(A\cap B)\subset (\cR+A)\cap(\cR+B) $, we have a strong form of subadditivity
\begin{equation}\label{sub-1}
|\cR+(A\cup B)|+|\cR+(A\cap B)|\le  |\cR+A|+|\cR+B|.
\end{equation}
Now clearly, for any fixed $A\subset \mathbb Z^d$ and $\cR^1,\ldots, \cR^N$ independent simple random walk ranges, the process defined by
$$L\Big(\prod_{i=1}^N \{n_i,\dots,m_i \}\Big) = \Big|\mathcal R^1[n_1,m_1]+\dots + \mathcal R^N[n_N,m_N]+A\Big|,$$
satisfies all the hypotheses of the previous theorem, and hence we get the almost sure existence of the limit below for any finite set $A$,
$$f_N(A) := \lim_{n\to \infty} \frac{\Big|\mathcal R^1_n+\dots + \mathcal R^N_n+A\Big|}{n^N} = \lim_{n_1,\dots,n_N\to \infty} \frac{\mathbb E\left[\Big|\mathcal R^1_{n_1}+\dots + \mathcal R^N_{n_N}+A\Big|\right]}{n_1\dots n_N}. $$
Note also that it follows immediately from this definition that $f_N$ satisfies for all $N$ the strong subadditivity property which reads
\begin{equation*}
f_N(A\cup B)+f_N(A\cap B)\le f_N(A)+f_N(B).
\end{equation*}
Therefore almost surely as well,
$$\lim_{n\to \infty} \frac{f_{N-1}(\mathcal R_n +A)}{n} = \lim_{n\to \infty} \frac{\mathbb E\big[f_{N-1}(\mathcal R_n +A)\big]}{n}. $$
Applying twice the dominated convergence theorem yields
\begin{equation}\label{rec-fN}
\begin{split}
f_N(A) & = \lim_{n\to \infty} \frac 1n\cdot \lim_{m\to \infty} \frac{\mathbb E\left[\Big|\mathcal R^1_m+\dots + \mathcal R^{N-1}_m + \mathcal R^N_n+A\Big|\right]}{m^{N-1}}\\
& =  \lim_{n\ge 1} \frac{\mathbb E[f_{N-1}(\mathcal R_n^N +A)]}{n} = \lim_{n\to \infty} \frac{f_{N-1}(\mathcal R_n^N +A)}{n} .
\end{split}
\end{equation}

In the next result we give an expression for the set-function $f_N(A)$ generalising the expression for capacity in the case when $N=1$. We give the proof in Section~\ref{sec-dual}.

\begin{lemma}\label{lem-dual}
Let $N$ be an integer, and consider dimension $d>2N$. Let $\cR^1,\ldots, \cR^N$ be independent ranges of double-sided simple random walks in $\Z^d$. Then,
\begin{equation}\label{char-limit2}
f_N(A)=\sum_{a\in A} \pr{\bigcap_{i=1}^{N}\big\{ ({\cR}^i(0,\infty)+\sum_{i<j\leq N}{\cR}^j(-\infty, \infty)+ a)\cap A =\emptyset\big\}}.
\end{equation}
\end{lemma}

\subsection{Functionals on trees}\label{sec-tree}
A fundamental property of the infinite invariant tree is its invariance in law after applying the shift on the labels. This 
fact was first observed by Le Gall and Lin~\cite{LGL16} 
on the restriction of the tree to $\mathcal T_+$ (including the root), 
and then on the full tree independently by Zhu~\cite{Zhu18} and Bai and Wan~\cite{BW20}. The infinite tree
has an invariant product measure, and the shift is actually a reversible map.
Le Gall and Lin introduced the infinite tree to be able to use ergodic theory 
to prove asymptotics for the size of the 
range of the first $n$ labelled sites of the invariant branching random walk, $\cT^0_n$.
They then transferred their result to critical trees conditioned 
on having total population $n$, as $n$ goes to infinity. 
Here, for simplicity, we only discuss $\cT^0_n$, 
and the reader is referred to ~\cite{LGL16} to transfer the
results to critical branching random walks conditioned to have population $n$.

In fact by following the same argument as in ~\cite{LGL16}
we obtain an ergodic limit for the sausage obtained by rolling a finite set $A$ over $\cT^0_n$
(i.e. the Minkowski sum of $A$ and $\cT^0_n$).
More precisely, for any finite set $A\subset \mathbb Z^d$, we obtain
$$
\frac{|\cT^0_n+A|}{n}\  \xrightarrow[n\to \infty]{(\mathbb P)} \ 
\sum_{x\in A} \mathbb P(\mathcal T_{+}^x\cap A= \emptyset).
$$ 
This latter expression turns out to be the branching capacity of $A$. 
Indeed, Zhu~\cite[Proposition 8.1]{Zhu16} showed that  
\begin{equation}\label{zhu-reversible}
\textrm{BCap}(A) = \sum_{x\in A} \mathbb P(\mathcal T_-^x\cap A= \emptyset)=
\sum_{x\in A} \mathbb P(\mathcal T_{+}^x \cap A = \emptyset).
\end{equation}
Thus, the original part in Theorem~\ref{theo-2saus} is to make the link 
with the set-function $f_2(A)$ obtained with two independent random walks from~\eqref{intro-6}. 

Now let us mention some natural extensions of our results. 
We can indeed deduce that also $\textrm{cap}(\cT^0_n+A)/n$ converges in probability,
and furthermore that the limit is of order $\text{Cap}_{d-6}(A)$. 
Let us just explain the proof in this case. First, applying twice the multi-parameter ergodic theorem, Theorem~\ref{thm:akcoglu-krengel}, 
and using Theorems~\ref{theo-2saus} and~\ref{thm.fAN}, we get that almost surely
\begin{align}\label{cap.Rn*}
& \lim_{n\to \infty} \frac{\textrm{cap}(\mathcal T_n^0+A)}{n} = \lim_{n\to \infty} \frac{|\mathcal R_n + \mathcal T_n^0+A|}{n^2} =  \lim_{n\to \infty} \frac{\textrm{BCap}(\mathcal R_n +A)}{n}\\
\nonumber & \asymp  \lim_{n\to \infty}\frac{f_2(\mathcal R_n +A)}{n} =  \lim_{n\to \infty} \frac{|\mathcal R_n^1+ \mathcal R_n^2 +\mathcal R_n^3 +A|}{n^3}=f_3(A) \asymp \text{Cap}_{d-6}(A),
\end{align}
where $\mathcal R_n$ and $(\mathcal R_n^i)_{i=1,2,3}$ are independent ranges of simple random walks, independent of $\mathcal T^0$.

\subsection{Dual representation for $f_N$}\label{sec-dual}

In this section we give the proof of Lemma~\ref{lem-dual} which again makes use of Theorem~\ref{thm:akcoglu-krengel}.

\begin{proof}[\bf Proof of Lemma~\ref{lem-dual}]

We can write 
\begin{align*}
	|\cR^1_n + \ldots + \cR^N_n +A | = \sum_{i_1\leq n}\sum_{x_1\in \cR^2_n + \ldots + \cR^N_n +A} \1(X_{i_1}^1 + x_1\notin \cup_{j_1>i_1} (X_{j_1}^1 + \cR^2_n + \ldots + \cR^N_n +A)) \\ = \sum_{i_1,i_2\leq n}\sum_{x_2\in \cR^3_n + \ldots + \cR^N_n +A} \1(X_{i_1}^1 + X_{i_2}^2 + x_2\notin \cup_{j_1>i_1} (X_{j_1}^1 + \cR^2_n + \ldots + \cR^N_n +A)) \\ \times \1(X_{i_2}^2+x_2\notin \cup_{j_2>i_2} (X_{j_2}^2 + \cR^3_n + \ldots + \cR^N_n +A)).
\end{align*}
Iterating this, we  obtain
	\begin{align*}
		|\cR^1_n + \ldots + \cR^N_n &+A |\\ = \sum_{a\in A} \sum_{i_1,\ldots, i_N\leq n} &\1(X_{i_N}^N +a \notin \cup_{j_N>i_N} (X_{j_N}^N + A)) \times \1(X_{i_{N-1}}^{N-1} + X_{i_{N}}^N + a \notin \cup_{j_{N-1}>i_{N-1}} (X_{j_{N-1}}^{N-1} + \cR^N_n +A)) \\ &\times \cdots
		\times \1(X_{i_{1}}^{1} +\ldots + X_{i_{N}}^N + a \notin \cup_{j_{1}>i_{1}} (X_{j_{1}}^{1} + \cR_n^{2} + \ldots +\cR^N_n +A)). 
		\end{align*}
		For all $i\in \{1,\ldots, N\}$ and $k\leq n$ we set 
		\[
		\til{\cR}^i(-k,n-k) = \bigcup_{k\leq j\leq n}\{ X^i_{k}- X^i_j\} \cup \bigcup_{0\leq j<k} \{X^i_k-X^i_j\} \ \text{ and }\ \til{\cR}^i(0,n-k) = \bigcup_{k\leq j\leq n} \{X_k^i - X_j^i\}.
		\]
		Then we can rewrite the expression above as 
		\begin{align*}
		|\cR^1_n + \ldots + \cR^N_n &+A |\\ 
		= \sum_{a\in A} \sum_{i_1,\ldots, i_N
		\leq n} &\1((\til{\cR}^N(0,n-i_N) + a) \cap A =\emptyset)\1((\til{\cR}^{N-1}(0,n-i_{N-1}) + \til{\cR}^N(-i_N,n-i_N) +a )\cap A=\emptyset ) \\&\times \cdots \times  \1((\til{\cR}^1(0,n-i_1) + \til{\cR}^2(-i_2, n-i_2) + \cdots \til{\cR}^N(-i_N,n-i_N)+ a) \cap A =\emptyset).
		\end{align*}
		Restricting the sum above over all $i_1, \ldots, i_N \in (\log n, n-\log n)$, dividing through by $n^N$ and applying Theorem~\ref{thm:akcoglu-krengel} we deduce that almost surely as $n\to\infty$
	\begin{align*}
		\frac{1}{n^N}\cdot |\cR^1_n + \ldots + \cR^N_n +A | \to \sum_{a\in A} \pr{\bigcap_{i=1}^{N}\big\{ ({\cR}^i(0,\infty)+\sum_{i<j\leq N}{\cR}^j(-\infty, \infty)+ a)\cap A =\emptyset\big\}},
	\end{align*}
	where ${\cR}(-\infty,\infty)$ corresponds to the range of a double-sided simple random walk. This now concludes the proof.
\end{proof}



\section{Preliminaries on local times}\label{sec-localtimes}
Our goal in this section is to prove Lemma~\ref{lem.second.moment}.
\subsection{Preliminaries}\label{subsec.prelim}

Fix an integer $N\ge 1$, and assume that $d\ge 1+2N$. 
Recall that  $X^1,\ldots, X^N$, are independent simple random walks on 
$\mathbb Z^d$, starting from the origin, and recall also \reff{def-localN}.   
Recall furthermore that $g(x)=G_1(x)=\mathbb E[\sum_{n=0}^\infty \1(X_n=x)]$, and for any $k\ge 2$,
 \[
 G_k(x-z) =\E{\ell_{z+X^1+\dots+X^N}(x)}= (G_{k-1}\ast g)(x-z),
 \]
 where $\ast$ stands for the convolution operator. Recall that $g(x) \asymp \tfrac{1}{(1+\|x\|)^{d-2}}$, which by an immediate induction shows that for any $k \in\{1,\dots,N\}$ (and as long as $d\ge 1+2N$), one has 
 \begin{equation}\label{asymp.Gk}
 G_k(x) \asymp \frac 1{(1+\|x\|)^{d-2k}}. 
 \end{equation}
 
\subsection{Proof of Lemma~\ref{lem.second.moment}}

Let $k,\ell,m\in \{1,\dots,N\}$, and for $a,b, z\in \Z^d$, let 
 	\[
 	F_{k,\ell,m}(z,a,b) = \sum_w G_{k}(z-b+w) G_{\ell}(w) G_m(a-b+w).
 	\]
 The first step towards the proof of Lemma~\ref{lem.second.moment} is the following claim. 
 \begin{claim}\label{cl:smallterms}
\emph{One has 
 	\[
 	\sum_{y,y'} (g(y) g(y-y') + g(y')g(y-y')) F_{k,\ell,m}(a+y,b+y') = F_{k+1,\ell+1,m}(z,a,b)+ F_{k+1,\ell,m+1}(z,a,b).
 	\]}
 \end{claim}
 
 \begin{proof}[\bf Proof]
 First of all we notice that $F_{k,\ell,m}(z,a,b) = F_{k,m,\ell}(z,b,a)$. Therefore it suffices to prove that 
 \[
 \sum_{y,y'} g(y) g(y-y') F_{k,\ell,m}(z,a+y,b+y') = F_{k+1,\ell+1,m}(z,a,b).
 \]
  We have 
 \begin{align*}
 	&\sum_{y,y'} g(y) g(y-y') F_{k,\ell,m}(z,a+y,b+y') \\
 	&= \sum_{y,y',w} g(y) g(y-y') G_{k}(z-b-y'+w) G_{\ell}(w) G_m(a+y-b-y'+w)\\
 	&=\sum_{y,u,w} g(y) g(u) G_k(z-b+u-y+w) G_\ell(w)G_m(a-b+u+w) \\
 	&=\sum_{u,w} g(u) G_\ell(w)G_m(a-b+u+w) G_{k+1}(z-b+u+w) \\
 	&= \sum_{u,v} g(u) G_\ell(v-u) G_m(a-b+v) G_{k+1}(z-b+v)\\
 	&= \sum_{v}G_{\ell+1}(v)G_m(a-b+v) G_{k+1}(z-b+v)=F_{k+1,\ell+1,m}(z,a,b)
 \end{align*}
 and this completes the proof.
 \end{proof}
For $z,a,b\in \mathbb Z^d$, define now
$$
V_N(z,a,b) = \E{\ell_{z+X^1+\dots+X^N}(a)\ell_{z+X^1+\dots+X^N}(b)} .
$$ 
The next step is obtained by a simple induction on $N\ge 1$. 
\begin{lemma}\label{lem.second}
 	\emph{One has for any $z,a,b\in \Z^d$, 
 	\begin{align*}
 		V_N(z,a,b) &\leq \  G_{N}(z-b)G_{N}(a-b)+ \sum_{k=1}^{N-1}{N-1 \choose k-1} F_{N,N-k,k}(z,a,b) \\&\ \ +\ G_{N}(z-a)G_{N}(a-b)+ \sum_{k=1}^{N-1}{N -1\choose k-1} F_{N,N-k,k}(z,b,a),
 	\end{align*}
	with the convention that the two sums are zero when $N= 1$.}
 	 \end{lemma}

 \begin{proof}[\bf Proof]
Defining, 
 \[
 W_N(z,a,b) = G_{N}(z-b)G_{N}(a-b)+ \sum_{k=1}^{N-1}{N-1 \choose k-1} F_{N,N-k,k}(z,a,b), 
 \]
 the statement of the lemma then becomes
 \begin{align}\label{eq:generalineq}
 	V_{N}(z,a,b)\leq \ W_N(z,a,b) + W_N(z,b,a).
 \end{align}
 We will prove this by induction on $N$. For $N=1$ 
 we get 
 \begin{align*}
 	V_N(z,a,b) & \le \mathbb E\Big[\sum_{s,t\in \mathbb N }\1(z+X_t=a)\cdot \1(z+X_{s+t}=b)\Big]+ \mathbb E\Big[\sum_{s,t\in \mathbb N}\1(z+X_{s+t}=a)\cdot \1(z+X_t=b)\Big] \\
	& = g(z-a)g(a-b)+g(z-b)g(a-b),
 \end{align*}
and hence~\eqref{eq:generalineq} holds for $N=1$. Suppose now that~\eqref{eq:generalineq} holds for $N$. We will establish it also for $N+1$. Summing over all the possible locations of the $(N+1)$-st walk and using the induction hypothesis gives
\begin{align*}
	V_{N+1}(z,a,b) \ &\leq \ \sum_{y,y'} (g(y)g(y-y')+g(y') g(y-y')) V_{N}(z,a+y,b+y')\\ 
	&\leq \ \sum_{y,y'} (g(y)g(y-y')+g(y') g(y-y'))\cdot \big(W_{N}(z,a+y,b+y')+W_{N}(z,b+y',a+y)\big).
\end{align*}
To simplify notation we let $A$ be the operator given by 
\[
Af(z,a,b) = \sum_{y,y'} (g(y)g(y-y') + g(y')g(y-y')) f(z,a+y,b+y')
\]
for any function $f$. 
To prove the lemma it thus suffices to show that 
\begin{align}\label{eq:maingoalwithF}
AW_N(z,a,b) = W_{N+1}(z,a,b).
 \end{align}
Note first that 
$$\sum_{y,y'} g(y')g(y-y') G_N(z-b-y')G_N(a-b+y-y') = G_{N+1}(z-b) G_{N+1}(a-b),$$
which gives the first term in $W_{N+1}(z,a,b)$. Note also that 
$$\sum_{y,y'} g(y)g(y-y') G_N(z-b-y')G_N(a-b+y-y') = F_{N+1,1,N}(z,a,b).$$
By Claim~\ref{cl:smallterms} we obtain
\[
A F_{N,1,N-1}(z,a,b) = F_{N+1,1,N}(z,a,b) + F_{N+1,2,N-1}(z,a,b),
\]
which shows that the coefficient of the term $F_{N+1,1,N}(z,a,b)$ in $AW_{N}(z,a,b)$ is given by $1+ {N-1 \choose N-2}$ which is equal to ${N\choose N-1}$. Thus the term $F_{N+1,1,N}(z,a,b)$ appears with the same coefficient in both $W_{N+1}(z,a,b)$ and $AW_{N}(z,a,b)$. 
Let $k\in \{2,\ldots, N-1\}$. Using Claim~\ref{cl:smallterms} again we get that 
\begin{align*}
	AF_{N,N-k,k}(z,a,b) &= F_{N+1,N+1-k,k}(z,a,b) + F_{N+1,N-k,k+1}(z,a,b) \quad \text{ and } \\
	AF_{N,N+1-k,k-1}(z,a,b) &= F_{N+1,N+1-k,k}(z,a,b) + F_{N+1,N+2-k,k-1}(z,a,b).
\end{align*}
We thus see that for $k\in \{2,\ldots, N-1\}$ the coefficient of the term $F_{N+1,N+1-k,k}(z,a,b)$ in $AW_{N}(z,a,b)$ is equal to 
\[
{N-1 \choose k-1} + {N-1\choose k-2} = {N\choose k-1},
\] 
which is the same as its coefficient in $W_{N+1}(z,a,b)$.
Using Claim~\ref{cl:smallterms} for a last time we see that the term $F_{N+1,N,1}(z,a,b)$ is one of the two terms of $AF_{N,N-1,1}(z,a,b)$, and hence its coefficient in $AW_{N}(z,a,b)$ must be ${N\choose 0}={N-1\choose 0} = 1$. Therefore, we see that the coefficients of all the terms appearing in $AW_N(z,a,b)$ and $W_{N+1}(z,a,b)$ are equal and this completes the proof of~\eqref{eq:maingoalwithF}.
\end{proof}

Finally we shall need the following claim. 

\begin{claim}\label{claim.second}\emph{
	Let $N\ge 1$ and $d>2N$. There exists $C>0$, such that for all $k\in \{1,\ldots,N-1\}$, and all $z,a,b\in \Z^d$, with $\|z\|\geq 2\max(\|a\|,\|b\|)$, 
	\[
	F_{N,N-k,k}(z,a,b) \le C\cdot G_{N}(z) G_{N}(a-b).
	\]}
\end{claim}

\begin{proof}[\bf Proof]
First of all note that for all $\ell, m$, such that $\ell + m\le N$, 
\[
G_\ell \ast G_m = G_{\ell+m}.  
\]
Moreover, a change of variables gives 
$$F_{N,N-k,k}(z,a,b) = \sum_{u\in \mathbb Z^d} G_{N}(u)G_{N-k}(u+b-z) G_{k}(u+a-z). $$ 
 We then have for $\|z\|\ge 2\max(\|a\|,\|b\|)$, using~\eqref{asymp.Gk}, 
\begin{align*}
	&\sum_{\|u\|\geq \|z\|/4} G_{N}(u) G_{N-k}(u+b-z)G_{k}(u+a-z) \\ 
	&\lesssim G_{N}(z)\sum_{u}G_{N-k}(u+b-z)G_{k}(u+a-z) =  G_{N}(z) G_{N}(a-b).
\end{align*}
On the other hand, we also have using again~\eqref{asymp.Gk}, 
\begin{align*}
	&\sum_{\|u\|\le \|z\|/4} G_{N}(u) G_{N-k}(u+b-z)G_{k}(u+a-z) \\ &\lesssim G_{N-k}(z) G_{k}(z) \sum_{\|u\|\le \|z\|/4} G_{N}(u)  
	\asymp \frac{\|z\|^{2N}}{1+\|z\|^{2d-2N}}  \asymp G_{N}(z)^2  \lesssim G_{N}(z) G_{N}(a-b)
\end{align*}
and this finishes the proof.
\end{proof}
The result follows now from a combination of Lemma~\ref{lem.second} and Claim~\ref{claim.second}.

\section{Hitting probabilities and capacities} \label{sec.KSd}
In this section we give the proof of Theorems~\ref{theo.IE} and~\ref{thm.KSd}. We start by giving the proof of Theorem~\ref{theo.IE} assuming Theorem~\ref{thm.KSd} and then give the proof of the latter for which we mainly follow the arguments of~\cite{KS99}, see also~\cite{K03}, which extends
the approach of \cite{FS89}.

\begin{proof}[\bf Proof of Theorem~\ref{theo.IE}]
	The proof follows from the combination of three distinct observations: (i) Theorem~\ref{thm.KSd}
with $N=2$, (ii) the hitting time asymptotics for the infinite invariant tree \reff{hitting-zhu} by Zhu
\cite{Zhu16}, and finally (iii) the fact that $\cpc{}{A}\asymp \cp{d-4}{A}$ proved in \cite{ASS23}.
\end{proof}

The rest of this section is devoted to the proof of Theorem~\ref{thm.KSd}.

Let $X^1,\ldots, X^N$ be i.i.d.\ simple random walks on $\Z^d$ started from $0$ with ranges $\cR_\infty^1, \ldots, \cR_\infty^N$ respectively. For $\gamma>0$ and a probability measure $\nu$ on $A$ we now define with a slight abuse of notation the~$\gamma$-energy of $\nu$   to be
\begin{align}\label{def:energy}
\mathcal E_\gamma(\nu) = \sum_{x,y\in \mathbb Z^d} (1+\|x-y\|)^{-\gamma} \, \nu(x)\nu(y). 	
\end{align}
\paragraph{Lower Bound.}
It suffices to prove that if $\nu$ is a probability measure on $A$, then 
\begin{align}\label{goal.lower}
\frac{\mathbb P\big((z+\mathcal R_\infty^1+ \dots+ \mathcal R_\infty^N)\cap A\neq \emptyset\big)}{G_{N}(z)} \gtrsim \frac{1}{\cE_{d-2N}(\nu)}, 
\end{align}
with an implicit constant that is independent of $\nu$.  
To this end, for any probability $\nu$ with support on $A$, let 
\[
Z_\nu= \sum_{a\in A} \nu(a) \cdot \ell_{z+X^1+\dots+X^N}(a).
\]
Then it is immediate that 
\begin{align}\label{eq:boundusingpayley}
	\mathbb P\big((z+\mathcal R_\infty^1+ \dots+ \mathcal R_\infty^N)\cap A\neq \emptyset\big) \geq \pr{Z_\nu>0} \geq \frac{(\E{Z_\nu})^2}{\E{Z_\nu^2}},
\end{align}
where for the last inequality we used  the Cauchy-Schwarz inequality.
For the first moment of $Z_\nu$, we have for any $z$ with $\|z\|\ge 2 \cdot\textrm{diam}(A)$, using~\eqref{asymp.Gk},
\[
\mathbb E[Z_\nu] = \sum_{a\in A} \nu(a) G_{N}(z-a) \gtrsim G_{N}(z). 
\]
For the second moment, by Lemma~\ref{lem.second.moment}
we have for $z$ with $\|z\|\ge 2 \cdot\textrm{diam}(A)$,
\begin{equation}\label{Z2}
	\E{Z^2_\nu} = \sum_{a,b\in A} \nu(a) \nu(b) V_N(z,a,b) \lesssim G_{N}(z)\sum_{a,b\in A}\nu(a) \nu(b) G_{N}(a-b) \lesssim G_{N}(z) \cE_{d-2N}(\nu).
\end{equation}
Plugging these two bounds into~\eqref{eq:boundusingpayley} yields~\eqref{goal.lower}.

\paragraph{Upper Bound.} We define $N$ random times, which are not stopping times.
\[
T_1=\inf\{t_1\geq 0: \exists \ t_2,\ldots, t_N \ \text{ s.t. } z+X_{t_1}^1+\ldots+X_{t_N}^N \in A\},
\]
and then inductively for $i=2,\ldots, N$
\[
T_i= \inf\{t_i\geq 0: \exists \ t_{i+1},\ldots, t_N \ \text{ s.t. } z+X_{T_1}^1+\ldots+X_{T_{i-1}}^{i-1}+X_{t_{i}}^i+\ldots + X_{t_N}^N \in A\}.
\]
We next define a probability measure $\mu$ on $A$ by setting for $a\in A$ 
\[
\mu(a) =\prcond{z+X_{T_1}^1+\ldots +X_{T_N}^N=a}{T_1<\infty}{}.
\]
It then suffices to prove that for $\|z\|$ sufficiently large 
\begin{equation}\label{goal.upperbound}
\frac{\pr{T_1<\infty}}{G_{N}(z)}\lesssim \frac{1}{\cE_{d-2N}(\mu)}.
\end{equation}

To this end we define the variable $Z_\mu=\sum_{a\in A} \mu(a)\ell_{z+X^1+\dots+X^N}(a)$,
and if $(\cF^i_n)_{n\geq 0}$ stands for the natural filtration of the walk $X^i$, 
then we define the multi-parameter process
\begin{align}\label{eq:productformofm}
\begin{split}
&M(t_1,\ldots ,t_N) = \econd{Z_\mu}{\F_{t_1}^1\otimes \ldots \otimes \F_{t_N}^N} \\
&= \sum_{a\in A} \mu(a)  \sum_{s_1,\ldots, s_N} \sum_{x_1,\ldots, x_{N-1}}  \prcond{X_{s_1}^1=x_1}{\F_{t_1}^1}{}\cdots \prcond{X_{s_N}^N=a-z-x_1-\ldots -x_{N-1}}{\F_{t_N}^N}{}
\end{split}
\end{align}
We then have almost surely for all $t_1,\ldots,t_N\in \N$ 
\begin{align*}
	&M(t_1,\ldots, t_N) \\ 
	&\geq \sum_{a\in A}\mu(a) \sum_{s_1\geq t_1,\ldots,s_N\geq t_N}\sum_{x_1,\ldots,x_{N-1}\in \Z^d} \prcond{X_{s_1}^1=x_1}{\F_{t_1}^1}{}\cdots \prcond{X_{s_N}^N=a-z-x_1-\ldots -x_{N-1}}{\F_{t_N}^N}{} \\
	&=\sum_{a\in A}\mu(a) \sum_{x_1,\ldots,x_{N-1}\in \Z^d} g(X_{t_1}^1-x_1) \cdots g(X_{t_{N-1}}^{N-1}-x_{N-1})g(z+x_1+\ldots+x_{N-1}+X_{t_N}^N-a)\\
	&= \sum_{a\in A}\mu(a) G_{N}(z+X_{t_1}^1+\ldots +X_{t_N}^N-a).
\end{align*}
Therefore, almost surely we get 
\begin{align*}
	\sup_{t_1,\ldots, t_N}M(t_1,\ldots,t_N)\ \geq \ \1(T_1<\infty) \sum_{a\in A}\mu(a) G_{N}(z+X_{T_1}^1+\ldots + X_{T_N}^N-a),
\end{align*}
and hence squaring both sides and taking expectations we obtain
\begin{align*}
	\E{\sup_{t_1,\ldots ,t_N}M^2(t_1,\ldots, t_N)}&\geq \econd{\left(\sum_{a\in A}\mu(a) G_{N}(z+X_{T_1}^1+\ldots+X_{T_N}^N-a) \right)^2}{T_1<\infty} \cdot \pr{T_1<\infty}\\
	& = \sum_{b\in A}\mu(b) \left(\sum_{a\in A}\mu(a)G_N(b-a) \right)^2 \cdot \pr{T_1<\infty}
	\gtrsim (\cE_{d-2N}(\mu))^2\cdot \pr{T_1<\infty},
	\end{align*}
where in the last step we used the Cauchy-Schwarz inequality and~\eqref{asymp.Gk}. 

By monotone convergence it then suffices to prove that for any $u_1,\dots,u_N$, 
\begin{equation}\label{mart.square}
\E{\sup_{t_1\le u_1,\ldots, t_N\le u_N}M^2(t_1,\ldots, t_N)}\lesssim G_{N}(z) \cE_{d-2N}(\mu), 
\end{equation}
with an implicit constant that is independent of $u_1,\dots,u_N$. This together with the inequality above would conclude the proof of~\eqref{goal.upperbound}.

Using the product expression for $M$ from~\eqref{eq:productformofm}, it is easy to check that the process 
$$t_1\mapsto \sup_{t_2\leq u_2,\dots,t_N\le u_N} M(t_1,t_2,\dots,t_N),$$ is a non-negative submartingale with respect to the filtration $(\F_{t_1}^1\vee \F_{u_2}^2\vee \dots \vee \F_{u_N}^N)_{t_1\geq 0}$.
Applying Doob's $L^2$-inequality we then deduce
\begin{align*}
\E{\sup_{t_1\leq u_1}\left(\sup_{t_2\leq u_2,\dots,t_N\le u_N} M^2(t_1,t_2,\dots,t_N)\right)} \leq 4\cdot  \E{\sup_{t_2\leq u_2,\dots,t_N\le u_N} M^2(u_1,t_2,\dots,t_N)}.
\end{align*}
Repeating the same argument $N$ times, gives 
$$\mathbb E\Big[  \sup_{t_1\le u_1, \dots,t_N\le u_N} M^2(t_1,\dots,t_N)\Big] \leq 4^N\cdot  \E{M^2(u_1,\dots,u_N)} \le 4^N\cdot \mathbb E[Z^2_\mu], $$
where for the final inequality we used Jensen's inequality. By~\eqref{Z2} for $z$ with $\|z\|\ge 2 \cdot\textrm{diam}(A)$ we get 
\[
\E{Z^2_\mu}\lesssim G_{N}(z) \cE_{d-2N}(\mu).
\]
Altogether this proves~\eqref{mart.square} and thus completes the proof. 


\section{Hitting probabilities and ergodic limits} \label{sec.charact}
In this section, we prove  Proposition~\ref{thm.charact}.
The proof is divided in two parts. First we prove~\eqref{hit.fA} in Section~\ref{sec.hit.fA}, which is the easiest direction, and then~\eqref{hit.fA.bis} in Section~\ref{sec.hit.fA.bis}, which is slightly more demanding. 

\subsection{Proof of~\eqref{hit.fA}.}\label{sec.hit.fA}
Let $N\ge 2$, and assume that $d>2N$. Recall the definition of the functions $G_N$ from the beginning of Section~\ref{subsec.prelim}. 
Since $\mathcal R_\infty$ and $-\mathcal R_\infty$ have the same law, it amounts to proving that
\begin{align*}
	\liminf_{\|z\|\to\infty} \frac{\pr{(z+\cR_\infty^1+\dots+\cR_\infty^{N-1})\cap (\cR_\infty +A)\neq \emptyset}}{G_N(z)} \gtrsim \widehat f_{N}(A), 
\end{align*}
where $\mathcal R_\infty$ is the range of a random walk $(X_k)_{k\ge 0}$, which is independent of $\mathcal R_\infty^1,\dots, \mathcal R_\infty^{N-1}$, and 
$$ \widehat f_{N}(A) = \lim_{n\to \infty}  \frac{\cp{d-2(N-1)}{\mathcal R_n^1+A}}{n}.$$ 
Let $\epsilon>0$ and let $\tau_r=\inf\{k\geq 0: X_k \notin B(0,r)\}$. 
Then we get for $\|z\|$ large enough, using Theorem~\ref{thm.KSd} for the second inequality, 
\begin{align*}
	& \frac{\pr{(z+\cR_\infty^1+\dots+\mathcal R_\infty^{N-1})\cap (\cR_\infty +A)\neq \emptyset}}{G_N(z)}  \\
	& \geq \frac{\pr{(z+\cR_\infty^1+\dots+\mathcal R_\infty^{N-1})\cap (\cR[0,\epsilon \|z\|^2] +A)\neq \emptyset, \tau_{\|z\|/2}>\epsilon \|z\|^2}}{G_N(z)} \\ 
	& \gtrsim \frac{G_{N-1}(z)}{G_N(z)}\cdot \E{\text{Cap}_{d-2(N-1)}(\cR[0,\epsilon\|z\|^2]+A)\cdot \1(\tau_{\|z\|/2}>\epsilon \|z\|^2)}  \\ 
	& \gtrsim \frac{1}{\|z\|^2}\cdot \E{\text{Cap}_{d-2(N-1)}(\cR[0,\epsilon\|z\|^2]+A)}\\
	&\quad -\frac{1}{\|z\|^2}\cdot \E{\text{Cap}_{d-2(N-1)}(\cR[0,\epsilon\|z\|^2]+A)\cdot \1(\tau_{\|z\|/2}\le \epsilon \|z\|^2)}. 
\end{align*}
As $\|z\|\to\infty$ we have that 
\begin{align*}
\frac{1}{\|z\|^2}\cdot \E{\text{Cap}_{d-2(N-1)}(\cR[0,\epsilon\|z\|^2]+A)} \to \epsilon \widehat f_{N}(A).
\end{align*}
Indeed, the convergence holds in $L^1$
since the sequence $(\tfrac{\text{Cap}_\gamma(\mathcal R_n+A)}{n})_n$ is uniformly bounded by some deterministic constant, for any $\gamma>0$.  Furthermore, by Cauchy-Schwarz we get 
\begin{align*}
	&\E{\text{Cap}_{d-2(N-1)}(\cR[0,\epsilon\|z\|^2]+A)\cdot \1(\tau_{\|z\|/2}\le \epsilon \|z\|^2)}\\
	&\qquad\qquad 
	\leq \sqrt{\E{\text{Cap}_{d-2(N-1)}(\cR[0,\epsilon\|z\|^2]+A)^2} \pr{\tau_{\|z\|/2}\le \epsilon \|z\|^2}}.
\end{align*}
By a standard random walk estimate we get for a positive constant $c$ that 
\[
\pr{\tau_{\|z\|/2}\le \epsilon \|z\|^2}\leq \exp(-c/\epsilon).
\]
Using again that $(\tfrac{\text{Cap}_{d-2(N-1)}(\cR_n+A)}{n})_n$ is bounded we get that it also converges to $\widehat f_{N}(A)$ in $L^2$. 
Hence this gives for $\|z\|$ sufficiently large
\begin{align*}
	\E{\text{Cap}_{d-2(N-1)}(\cR[0,\epsilon\|z\|^2]+A)^2} \leq 2 \|z\|^4\cdot \widehat f_{N}(A)^2,
\end{align*}
using also that $\widehat f_{N}(A)$ is positive (since $d\ge 5$). Therefore we get 
\begin{align*}
	\E{\text{Cap}_{d-2(N-1)}(\cR[0,\epsilon\|z\|^2]+A)\cdot \1(\tau_{\|z\|/2}\le \epsilon \|z\|^2)} \lesssim \|z\|^2 \widehat f_{N}(A) \exp(-c/(2\epsilon)).
	\end{align*}
Putting everything together now gives that for $\|z\|$ sufficiently large
\begin{align*}
	\frac{\pr{(z+\cR_\infty^1+\dots+\mathcal R_\infty^{N-1})\cap (\cR_\infty +A)\neq \emptyset}}{G_N(z)} \gtrsim \epsilon \widehat f_{N}(A)  - \exp(-c/(2\epsilon)) \widehat f_{N}(A) \gtrsim \widehat f_{N}(A), 
	\end{align*}
by taking $\epsilon$ sufficiently small. This finishes the proof. \hfill $\square$


\subsection{Proof of~\eqref{hit.fA.bis}.}\label{sec.hit.fA.bis}
We use here the same notation as in Section~\ref{sec.hit.fA}. 
We define for $i\in \mathbb Z$, $r_i = 2^i\|z\|$, and let $I$ be the maximal index such that $r_{-I} \ge 4\,  \textrm{diam}(A)$. 
Now for $i\ge -I$, define 
$$\mathcal B_i = \partial B(z,r_{i+1}) \cup \partial B(z,r_{i-1}).$$ 
and for $i\ge -I$, and $k\ge 0$, let 
$$\tau_i^k = \inf \Big\{n\ge \sigma_i^{k-1} : X_n \in \partial B(z, r_i)\Big\}, \quad \textrm{and}\quad  
\sigma_i^k = \inf\Big\{n \ge \tau_i^k : X_n \in \mathcal  B_i\Big\},$$ 
with the convention $\sigma_i^{-1} =0$. To simplify notation we will also write $\tau_i = \tau_i^0$ and $\sigma_i = \sigma_i^0$, for $i\ge -I$. 
Note that by definition one has $\tau_0 = 0$. 
Then let for $i\ge -I$, 
$$\mathcal R_{(i)} = \bigcup_{k\ge 0} \mathcal R[\tau_i^k,\sigma_i^k].$$
Observe that on the event $\{\tau_{-I} = \infty\}$, one has 
$$\mathcal R_\infty = \bigcup_{i\ge -I} \mathcal R_{(i)}.$$
By splitting the set $\mathcal R_{(i)} +A$ into subsets of diameter $r_i/100$ each, we get for $i\geq -I$ using Theorem~\ref{thm.KSd}
\begin{align*}
	\mathbb P\big((z + \mathcal R_\infty^1+\dots+\mathcal R_\infty^{N-1}) \cap (\mathcal R_{(i)} + A)\neq \emptyset \big) \lesssim G_{N-1}(r_i) \cdot  \mathbb E\Big[ \text{Cap}_{d-2(N-1)}(\mathcal R_{(i)} +A)\Big].
\end{align*}
Using a union bound and the above we now get 
\begin{align}\label{upper.intersection}
& \nonumber \mathbb P\big((z + \mathcal R_\infty^1+\dots+\mathcal R_\infty^{N-1}) \cap (\mathcal R_\infty + A)\neq \emptyset \big) \\
& \nonumber \le \mathbb P(\tau_{-I}<\infty) + \sum_{i=-I}^\infty \mathbb P\big((z + \mathcal R_\infty^1+\dots+\mathcal R_\infty^{N-1}) \cap (\mathcal R_{(i)} + A)\neq \emptyset \big)  \\
& \lesssim \frac{g(z)}{g(\textrm{diam}(A))} + \sum_{i= -I}^\infty G_{N-1}(r_i)\cdot  \mathbb E\Big[ \text{Cap}_{d-2(N-1)}(\mathcal R_{(i)} +A)\Big].  
\end{align}
Now one has for any $-I\le i < 0$, using the transience of the walk, and writing $\mathbb E_x$ for the expectation with respect to the law of a simple random walk starting from $x$, 
\begin{align}\label{eq.Ri1}
\nonumber \mathbb E\Big[ \text{Cap}_{d-2(N-1)}(\mathcal R_{(i)} +A)\Big] & \le \mathbb P(\tau_i<\infty) \cdot \sup_{x\in \partial B(z,r_i)} 
\mathbb E_x\Big[ \text{Cap}_{d-2(N-1)}(\mathcal R_{(i)} +A)\Big]\\
\nonumber & \lesssim \frac{g(z)}{g(r_i)} \cdot \sup_{x\in \partial B(z,r_i)}  \Big(\sum_{k\ge 0}\mathbb P_x(\tau_i^k<\infty)\Big) \cdot 
\mathbb E_x\Big[ \text{Cap}_{d-2(N-1)}(\mathcal R[0,\sigma_i] +A)\Big]\\ 
& \lesssim \frac{g(z)}{g(r_i)} \cdot \sup_{x\in \partial B(z,r_i)}  \mathbb E_x\Big[ \text{Cap}_{d-2(N-1)}(\mathcal R[0,\sigma_i] +A)\Big]. 
\end{align}
Likewise for any $i\ge 0$, one has 
\begin{equation}\label{eq.Ri2}
\mathbb E\Big[ \text{Cap}_{d-2(N-1)}(\mathcal R_{(i)} +A)\Big] \lesssim \sup_{x\in \partial B(z,r_i)}  \mathbb E_x\Big[ \text{Cap}_{d-2(N-1)}(\mathcal R[0,\sigma_i] +A)\Big]. 
\end{equation} 
Now we claim that for any $i\ge -I$, one has 
\begin{equation}\label{claim.i}
\sup_{x\in \partial B(z,r_i)}  \mathbb E_x\Big[ \text{Cap}_{d-2(N-1)}(\mathcal R[0,\sigma_i] +A)\Big] \lesssim \widehat f_{N}(A) \cdot r_i^2. 
\end{equation}
Let us postpone the proof of the claim and conclude the proof of~\eqref{hit.fA.bis}. Plugging~\eqref{claim.i} into~\eqref{eq.Ri1} and~\eqref{eq.Ri2}, and using~\eqref{upper.intersection} we get using that $d\ge 1+2N$, 
$$\mathbb P\big((z + \mathcal R_\infty^1+\dots+\mathcal R_\infty^{N-1}) \cap (\mathcal R_\infty + A)\neq \emptyset \big)  \lesssim 
\frac{g(z)}{g(\textrm{diam}(A))} + \widehat f_{N}(A) \cdot G_N(z).$$
Dividing both sides by $G_N(z)$, and letting $\|z\|\to\infty$ concludes the proof of~\eqref{hit.fA.bis}.

Thus it only remains to prove the claim~\eqref{claim.i}. 
For this one can just write, using monotonicity of $\gamma$-capacities, for any $x\in \partial B(z,r_i)$, and some constant $c>0$, 
\begin{align*}
 \mathbb E_x\Big[ \text{Cap}_{d-2(N-1)} (\mathcal R[0,\sigma_i] +A) \Big] & = \sum_{k\ge 0}  \mathbb E_x\Big[ \text{Cap}_{d-2(N-1)}(\mathcal R[0,\sigma_i] +A) \cdot 
 \1( k r_i^2\le \sigma_i < (k+1)r_i^2)\Big] \\
 & \lesssim \sum_{k\ge 0}  \mathbb E_x\Big[ \text{Cap}_{d-2(N-1)}(\mathcal R[0,(k+1)r_i^2] +A) \cdot \1(\sigma_i \ge  k r_i^2)\Big] \\
 &  \lesssim \sum_{k\ge 0}  \mathbb E_x\Big[ \text{Cap}_{d-2(N-1)}(\mathcal R[0,(k+1)r_i^2] +A)^2\Big]^{1/2}  \cdot \mathbb P(\sigma_i \ge  k r_i^2)^{1/2} \\
 & \lesssim \widehat f_{N}(A) \cdot \sum_{k\ge 0} kr_i^2 \cdot \exp(-ck) \lesssim \widehat f_{N}(A) \cdot r_i^2, 
\end{align*}
as wanted, where for the penultimate bound we used again that $(\tfrac{\text{Cap}_{d-2(N-1)}(\cR_n+A)}{n})_n$ is bounded we get that it also converges to $\widehat f_{N}(A)$ in $L^2$. 
. \hfill $\square$

\section{Capacity of the sausage in $d=4$}\label{sec.d4}
In this Section, we probe Proposition~\ref{thm.d4}.
Recall that we assume here that $d=4$. Following the notation of~\cite{L91}, we set $a_4= 2/\pi^2$. 
\begin{claim}\label{cl:standardfromlawler}
\emph{	Fix $M>0$. There exists a positive constant $c$ and $n_0$ so that for all $n\geq n_0$, if $\xi$ is a geometric random variable of parameter $1/n$, then for all $x$ with $\|x\|\leq M$ we have for all $\epsilon>0$
	\[
	\pr{\left|\sum_{k=0}^{\xi}g(X_k-x)- 2a_4 (\log n)\right|\geq \epsilon \log n} \leq \frac{c}{\epsilon^2\log n}.
	\]}
\end{claim}

\begin{proof}[\bf Proof]
Let 
$$\Delta = \sum_{k=0}^{\xi}\big(g(X_k-x)-g(X_k)\big). $$ 
By the gradient estimate for Green's function (see e.g. Theorem 1.5.5 in~\cite{L91}), one has 
$$|g(X_k-x)- g(X_k)|\lesssim \frac{1}{1+\|X_k\|^{3}}.$$ 
Furthermore, a standard computation gives 
$$ \E{\sum_{k=0}^\infty \frac 1{1+\|X_k\|^3}} = \sum_{x\in \mathbb Z^4} \frac{g(x)}{1+\|x\|^{3}} <\infty. $$ 
Therefore Markov's inequality gives that 
$$\mathbb P\big(|\Delta|\ge \frac{\varepsilon}{2} \log n\big) \lesssim  \frac{1}{\varepsilon \log n}.$$
To conclude the proof we use the concentration results proved by Lawler. Indeed, he shows in Lemma 4.2.1 in~\cite{L91} that 
	\begin{align*}
		\pr{\left|\sum_{k=0}^{\xi}g(X_k)- 2a_4 (\log n)\right|\geq \frac{\epsilon}{2} \log n} \lesssim \frac{1}{\varepsilon^2 \log n}.
	\end{align*}
	\end{proof}
	
We now introduce some notation. Fix a set $A$ and let $X$ be a simple random walk with range $\cR$. Let $\til{\cR}$ be an independent range. Let $\xi_n^\ell$ and $\xi_n^r$ be two independent geometric random variables of parameter $1/n$. For every $x\in A$ we set 
\begin{align*}
	\mathcal{A}_n^x \quad &= \quad \1((x+\widetilde{\cR}_\infty)\cap (\cR[-\xi_n^{\ell},\xi_n^{r}]+A) = \emptyset) \\
	e_n^x\quad &=\quad  \1(x\notin (\cR[1,\xi_n^r]+A))\\
	\cU_n^x\quad  &=\quad  \sum_{y\in A} \sum_{-\xi_n^\ell\leq k\leq \xi_n^r} g(x,X_k+y).
\end{align*}

\begin{lemma}\label{thm:lawler}
	\emph{We have
	\[
	\sum_{x\in A} \E{\1(\cA_n^x)\cdot e_n^x \cdot \cU_n^x} = |A|.
	\]}
\end{lemma}	
	
	\begin{proof}[\bf Proof]
	
	For every nearest neighbour path $(x_1,\ldots, x_m)$ we define
	\[
	B(m,x_1,\ldots,x_m) = \{ \xi_n^\ell + \xi_n^r = m, \ X_{-\xi_n^\ell + k}-X_{-\xi_n^\ell} = x_k, \ \forall \ 1\leq k\leq m\}, 
	\]
	and for all $0\leq j\leq m$ we define
	\[
	B(m,j,x_1,\ldots,x_m) = \{\xi_n^\ell = j, \ \xi_n^r=m-j, \ X_{-\xi_n^\ell + k}-X_{-\xi_n^\ell} = x_k, \ \forall \ 1\leq k\leq m\}.
	\]
	Using the independence of the increments of the walk and the geometric random variables we then obtain 
	\[
	\prcond{B(m,j,x_1,\ldots,x_m)}{B(m,x_1,\ldots, x_m)}{} = \frac{1}{m+1}.
	\]
	Setting $x_0=0$, we then have 
	\begin{align*}
		\sum_{x\in A} &\E{\1(\cA_n^x)\cdot e_n^x \cdot \cU_n^x} \\
		=& \sum_{x\in A} \sum_{m=0}^{\infty}\sum_{(x_1,\ldots,x_m)} \frac{\pr{B(m,x_1,\ldots,x_m)}}{m+1}\cdot \sum_{k=0}^{m} \sum_{j=0}^{m}  \1(x+x_j\notin (\{x_{j+1},\ldots,x_m\}+A)) \\
		&\times \pr{(x+x_j+\til{\cR}_\infty)\cap (\{x_0,x_1,\ldots,x_m \}+A)=\emptyset}\sum_{y\in A} g(x+x_j-x_k-y).
	\end{align*}
		Using the last exit decomposition formula to the set $\{x_0,\ldots,x_m\}$  and the starting point $x_k+y$ we get 
				\begin{align*}
		1= \sum_{x\in A}\sum_{j=0}^{m}\1(x+x_j&\notin (\{x_{j+1},\ldots, x_m\}+A))\\ &\times \pr{(x+x_j+\til{\cR}_\infty)\cap (\{x_0,x_1,\ldots,x_m \}+A)=\emptyset} g(x+x_j-x_k-y) .
		\end{align*}
		Substituting this above we obtain
		\begin{align*}
			\sum_{x\in A} \E{\1(\cA_n^x)\cdot e_n^x \cdot \cU_n^x} = \sum_{m=0}^{\infty} \sum_{(x_1,\ldots,x_m)}\pr{B(m,x_1,\ldots,x_m)}\cdot |A|= |A|, 
		\end{align*}
		and this concludes the proof.
	\end{proof}

	\begin{lemma}\label{lem:sumoverA}
\emph{We have 
\[
\sum_{x\in A} \E{\1(\cA_n^x)\cdot e_n^x} = (1+o(1))\cdot \frac{|A|}{4a_4\log n}. 
\]	}	
	\end{lemma}
	
\begin{proof}[\bf Proof]
We have from Lemma~\ref{thm:lawler} that 
\[
\sum_{x\in A} \E{\1(\cA_n^x)\cdot e_n^x \cdot \cU_n^x} = |A|.
\]
We now get 
\begin{align*}
	\sum_{x\in A}\E{\1(\cA_n^x)\cdot e_n^x }&=  \frac{|A|}{4a_4\log n} + \frac{1}{4a_4\log n}\cdot \sum_{x\in A} \E{\1(\cA_n^x)\cdot e_n^x \cdot (4a_4\log n -\cU_n^x)}. 
\end{align*}
For every $x\in A$ and $\epsilon>0$ we let 
\[
B_x = \{|\cU_n^x - \E{\cU_n^x}|\geq \epsilon \log n \}.
\] 
Then we have 
\begin{align*}
	\E{\1(\cA_n^x)\cdot e_n^x\cdot |\E{\cU_n^x}-\cU_n^x| } \leq \epsilon\log n \cdot \E{\1(\cA_n^x)\cdot e_n^x}  + \E{\1(\cA_n^x)\cdot |\E{\cU_n^x} - \cU_n^x| \cdot \1(B_x)}.
\end{align*}
We now explain that it suffices to prove that 
\begin{align}\label{eq:maingoal}
	\E{\1(\cA_n^x)\cdot |\E{\cU_n^x} - \cU_n^x| \cdot \1(B_x)} \lesssim \frac{1}{(\log n)^{3/2}}.
\end{align}
	Indeed, once this is established, then we get 
	\begin{align*}
		\left|\sum_{x\in A}\E{\1(\cA_n^x)\cdot e_n^x } -  \frac{|A|}{4a_4(\log n)}\right|\leq \epsilon\sum_{x\in A} \E{\1(\cA_n^x) \cdot e_n^x} + \mathcal O\left(\frac{1}{(\log n)^{3/2}}\right), 
 	\end{align*}
	and since this holds for any $\varepsilon>0$, this concludes the proof.
	So we now turn to prove~\eqref{eq:maingoal}. By the Cauchy-Schwarz inequality we obtain
	\begin{align*}
		\E{\1(\cA_n^x)\cdot |\E{\cU_n^x} - \cU_n^x| \cdot \1(B_x)} \leq \sqrt{\pr{\cA_n^x\cap B_x} \cdot \E{(\E{\cU_n^x} - \cU_n^x)^2}} \leq  \sqrt{\pr{\cA_n^x\cap B_x} \cdot \log n}
	\end{align*}
	using Lawler's estimate for the last inequality. It remains to bound the last probability appearing above. To do this we define 
	\[
	\cU_n^{x,1} = \sum_{k=-\xi_n^\ell}^{0} g(x,X_k) \quad \text{and}\quad  \cU_n^{x,2} = \sum_{k=0}^{\xi_n^r} g(x,X_k),
	\]
	and also two events for $i=1,2$
	\[
	B_x^i = \{|\cU_n^{x,i}-2a_4\log n| \geq \epsilon\log n/4\}.
	\]
	Then it is clear that $B_x\subseteq B_x^1\cup B_x^2$, at least for $n$ large enough, and we have 
	\begin{align*}
		\pr{\cA_n^x\cap B_x}&\leq \pr{(x+\til{\cR}_\infty)\cap (\cR[-\xi_n^\ell,0]+A)=\emptyset, B_x^2} + \pr{(x+\til{\cR}_\infty)\cap (\cR[0,\xi_n^r]+A)=\emptyset, B_x^1} \\
		&= 2\pr{(x+\til{\cR}_\infty)\cap (\cR[-\xi_n^\ell,0]+A)=\emptyset}\pr{B_x^2}.
	\end{align*}
	Since $x\in A$ we get 
	\[
	\pr{(x+\til{\cR}_\infty)\cap (\cR[-\xi_n^\ell,0]+A)=\emptyset} \leq \pr{\til{\cR}_\infty\cap \cR[-\xi_n^\ell,0] = \emptyset} \lesssim \frac{1}{\sqrt{\log n}}, 
	\]
	using Corollary 3.7.1 in~\cite{L91} for the last inequality. 
	By Claim~\ref{cl:standardfromlawler} we get that 
	\[
	\pr{B_x^2}\lesssim \frac{1}{\log n},
	\]
	and hence altogether this gives 
	\begin{align*}
		\E{\1(\cA_n^x)\cdot |\E{\cU_n^x} - \cU_n^x| \cdot \1(B_x)} \lesssim \frac{1}{(\log n)^{3/2}}, 
	\end{align*}
	and this concludes the proof.
\end{proof}

	\begin{proof}[\bf Proof of Proposition~\ref{thm.d4}]	 
	
	We have 
	\begin{align*}
		\E{\textrm{cap}(\mathcal R_n+A)} 
		= \sum_{x\in A} \sum_{j=0}^{n} \pr{x\notin (\cR[1,n-j]+A), (x+\til{\cR}_\infty) \cap (\cR[-j,n-j]+A)=\emptyset}.
 	\end{align*}
	We then get the following bounds for $m=n/(\log n)^2$
	\begin{align*}
		&\E{\textrm{cap}(\mathcal R_n+A)} \geq n\cdot \sum_{x\in A} \pr{x\notin (\cR[1,n]+A),  (x+\til{\cR}_\infty) \cap (\cR[-n,n]+A)=\emptyset} \text{ and } \\
		&\E{\textrm{cap}(\mathcal R_n+A)} \leq m\cdot |A| +(n-m) \cdot \sum_{x\in A}\pr{x\notin (\cR[1,m]+A),  (x+\til{\cR}_\infty) \cap (\cR[-m,m]+A)=\emptyset}.
			\end{align*}
	It is then easy to conclude using Lemma~\ref{lem:sumoverA}.
	\end{proof}

 \section{Open problems}\label{sec-open}
We discuss some open problems related to our present analysis.

\paragraph{Hitting Times.} In Theorem~\ref{thm.KSd} we established that the probability that a sum of $N$ simple random walks started from $z$ hits a finite set $A$ is of order $G_N(z)\cdot \text{Cap}_{d-2N}(A)$, when $d>2N$ and $\|z\|\to \infty$. A natural question is whether the quantity 
\[
\frac{1}{G_N(z)} \bP\big( z+\cR^1_\infty+
\dots+\cR^N_\infty \cap A\not= \emptyset\big)
\]
has a limit as $\|z\|\to\infty$. 

Another natural question is whether the analogue of Theorem~\ref{thm.KSd} holds for the sum of invariant trees. More precisely, when $d>4N$ and $A$ is a finite subset of $\Z^d$, is the quantity
\[
\frac{1}{G_{2N}(z)} \bP\big( z+\cT^1_\infty+
\dots+\cT^N_\infty \cap A\not= \emptyset\big)
\]
of order $\cp{d-4N}{A}$ as $\|z\|\to \infty$?
One difficulty here will be to be able to define hitting times of the set $A$ for which we can decouple future and past for each invariant tree. 

\paragraph{Tails of local times.} In Theorem~\ref{theo.IE} we stated an intersection equivalence between the sum of two simple random walks and an infinite invariant tree. However, the equivalence between these two processes is not expected to hold beyond hitting probabilities. Obtaining tails for the local times of additive walks is an open problem. We expect the local times of the sum of two walks to decay as a stretched exponential when $d>4$ as opposed to an exponential decay in the case of the invariant tree (see~\cite[Theorem~1.6]{ASS23}).

\paragraph{Critical models.} There is a range of critical models for which several questions arise: the Minkowski sum of two simple random walks in $d=4$, the Minkowski sum of $3$ walks in $d=6$ and so on. Interesting questions include: 
\begin{itemize}
	\item [(i)] The fluctuations of the capacity of a sausage
obtained as we roll a
finite set over the trajectory of the process.
\item [(ii)] The tail of the local times, where we expect a stretched exponential
tail. It would be interesting to have a representation of
the rate function.
\item [(iii)] The {\it folding phenomenon} for additive
walks or trees, and the first estimates
we need is an upper bound on the probability to cover a region
up to a certain density (measured in a certain space-scale).
A typical example of such a folding phenomenon is the event of having a large
intersection between two invariant trees in dimension $d>8$, and
the approach should follow the analogous problem of intersection
of two random walks in $d>4$ studied recently in \cite{AS21}.
\end{itemize}


\begin{thebibliography}{99}
\bibitem[AK81]{AK81} M. A. Akcoglu, U. Krengel. Ergodic theorems for superadditive processes. J. Reine Angew. Math. 323 (1981), 53--67. 

\bibitem[Ald91]{Ald91} D. Aldous. Asymptotic fringe distributions for general families of random trees. Ann. Appl. Probab. 1 (1991), 228--266. 

\bibitem[AS21]{AS21} A. Asselah,; B. Schapira, Large deviations for intersections of random walks. Comm. Pure Appl. Math. 76 (2023), no. 8, 1531--1553.


\bibitem[ASS19]{ASS19} A. Asselah, B. Schapira, P. Sousi. Capacity of the range of random walk on $\mathbb Z^4$. Ann. Probab. 47 (2019), 1447--1497.

\bibitem[ASS23]{ASS23} A. Asselah, B. Schapira, P. Sousi. Local times and capacity for transient branching random walks. arXiv:2303.17572. 


\bibitem[BW22]{BW20}  T. Bai, Y. Wan. Capacity of the range of tree-indexed random walk. Ann. Appl. Probab. 32 (2022), 1557--1589.

\bibitem[BPP95]{BPP95} I. Benjamini, R. Pemantle, Y. Peres. Martin capacity for Markov chains. Ann. Probab. 23 (1995), 1332--1346. 





\bibitem[FS89]{FS89} P. J. Fitzsimmons, T. S. Salisbury. Capacity and energy for multiparameter Markov processes. Ann. Instit. H. Poinc. Probab. Stat. 25 (1989), 325--350. 




\bibitem[IMcK74]{IMcK74}
 K. Ito, H. P. McKean, Jr.
Diffusion processes and their sample paths. 
Springer-Verlag, Berlin, 1974.
Second printing, corrected, Die Grundlehren der mathematischen Wissenschaften, Band 125.



\bibitem[JO69]{JO69}  N. Jain, S. Orey. On the range of random walk. Israel J. Math. 6 (1968), 373--380 (1969).


\bibitem[K03]{K03} D. Khoshnevisan. Intersections of Brownian motions. Expo. Math. 21 (2003), 97--114. 


\bibitem[KS99]{KS99} D. Khoshnevisan, Z. Shi. Brownian sheet and capacity. Ann. Probab. 27 (1999), 1135--1159. 

\bibitem[K73]{K73} J. F. C. Kingman. Subadditive ergodic theory. Ann. Probability 1 (1973), 883--909. 

\bibitem[L91]{L91} G. F. Lawler. Intersections of random walks. Second edition, Birkhauser, 1996.



\bibitem[LG90]{LG90} J.-F. Le Gall. Wiener sausage and self-intersection local times. J. Funct. Anal. 88 (1990), 299--341. 


\bibitem[LGL16]{LGL16} J.-F. Le Gall, S. Lin. The range of tree-indexed random walk. 
J. Inst. Math. Jussieu 15 (2016), 271--317. 


\bibitem[P96]{P96} Y. Peres. Intersection-equivalence of Brownian paths and certain branching processes. Comm. Math. Phys. 177 (1996), 417--434.


\bibitem[Port65]{Port65} S. C. Port. Limit theorems involving capacities for recurrent Markov chains. J. Math. Anal. Appl. 12 (1965), 555--569. 


\bibitem[S96]{S96} T. S. Salisbury. Energy, and intersections of Markov chains. Random discrete structures (Minneapolis, MN, 1993), 213--225, IMA Vol. Math. Appl., 76, Springer, New York, 1996.



\bibitem[Sp64]{Sp64} F. Spitzer. 
Electrostatic capacity, heat flow, and Brownian motion. Z. Wahrscheinlichkeitstheorie und Verw. Gebiete 3 (1964), 110--121.

\bibitem[Sp73]{Sp73} F. Spitzer. Discussion on Professor Kingman's Paper. Ann. Probab. 1 (1973), 900--909. 

\bibitem[Zhu16]{Zhu16}  Q. Zhu. On the critical branching random walk I: branching capacity and visiting probability,  arXiv:1611.10324,


\bibitem[Zhu18]{Zhu18} Q. Zhu. Branching interlacements and tree-indexed random walks in tori. arXiv:1812.10858. 


\end{thebibliography}
\end{document}